\documentclass[11pt]{article}

\usepackage{amssymb}
\usepackage{amsmath}
\usepackage{amsthm}

\numberwithin{equation}{section}

\newtheoremstyle{slplain}
  {\topsep}
  {\topsep}
  {\slshape}
  {0pt}
  {\bfseries}
  {.}
  {0.5em}
  {}

\theoremstyle{slplain}
  \newtheorem{THM}{Theorem}[section]
  \newtheorem{LEM}[THM]{Lemma}
  \newtheorem{PROP}[THM]{Proposition}
  \newtheorem{COR}[THM]{Corollary}

\theoremstyle{definition}

\newcommand{\Fraisse}{Fra\"\i ss\'e}

\newcommand{\lexle}{\mathrel{\le_{\mathit{lex}}}}

\newcommand\nlongrightarrow{\longrightarrow\kern -1.45em/\kern 0.9em}

\renewcommand{\le}{\leqslant}

\newcommand{\0}{\varnothing}

\renewcommand{\phi}{\varphi}
\renewcommand{\epsilon}{\varepsilon}

\newcommand{\NN}{\mathbb{N}}

\newcommand{\RR}{\mathbb{R}}

\newcommand{\Boxed}[1]{\mbox{$#1$}}
\newcommand{\id}{\mathrm{id}}

\newcommand{\Set}{\mathbf{Set}}

\newcommand{\calA}{\mathcal{A}}
\newcommand{\calB}{\mathcal{B}}
\newcommand{\calC}{\mathcal{C}}

\newcommand{\calE}{\mathcal{E}}

\newcommand{\calG}{\mathcal{G}}

\newcommand{\calK}{\mathcal{K}}
\newcommand{\calL}{\mathcal{L}}

\newcommand{\calO}{\mathcal{O}}
\newcommand{\calP}{\mathcal{P}}
\newcommand{\calQ}{\mathcal{Q}}

\newcommand{\calS}{\mathcal{S}}
\newcommand{\calT}{\mathcal{T}}

\newcommand{\tp}{\mathrm{tp}}

\newcommand{\Emb}{\mathrm{Emb}}

\newcommand{\im}{\mathrm{im}}

\newcommand{\RSurj}{\mathrm{RSurj}}
\newcommand{\Map}{\mathrm{Map}}

\title{Big Ramsey combinatorics of the Cantor set\\
        and a simple proof of Blass' perfect set theorem}
\author{%
  Dragan Ma\v sulovi\'c\\
  University of Novi Sad, Faculty of Sciences\\
  Department of Mathematics and Informatics\\
  Trg Dositeja Obradovi\'ca 3, 21000 Novi Sad, Serbia\\
  e-mail: dragan.masulovic@dmi.uns.ac.rs
}

\begin{document}
\maketitle

\begin{abstract}
In this paper we present a simple approach to big Ramsey combinatorics of the Cantor set $2^\omega$.
Using Infinite Dual Ramsey Theorem of Carlson and Simpson, we show that $2^\omega$,
viewed as a topological space, has finite big Ramsey degrees. We then examine several natural topological
first-order structures arising from the Cantor set and prove that each of them inherits finite big Ramsey degrees.
As a consequence, we obtain a simple proof of Blass' perfect set theorem, although our method does not recover the sharp bound $(n-1)!$
for the number of colors. We also show that the complete Boolean algebra on countably many atoms has finite big Ramsey degrees,
in contrast with the recent result showing that the countable atomless Boolean algebra does not have big Ramsey degrees.

  \bigskip


  \bigskip

  \noindent
  \textbf{AMS Subj. Classification (2020):} 05C55, 03C68, 54F99.
\end{abstract}

\section{Introduction}

Structural Ramsey theory underwent a paradigm shift at the turn of the century with
the publication of the seminal work by Kechris, Pestov, and Todor\v cevi\'c \cite{KPT}.
This paper established what is now known as the \emph{KPT correspondence},
a profound link between the Ramsey property of a class of finite structures and the
topological dynamics of the automorphism group of its \Fraisse\ limit.
Notably, the KPT correspondence initiated several new research directions,
including a renewed and intensive interest in the study of big Ramsey degrees.
While the study of big Ramsey degrees is firmly rooted in the context of \Fraisse\ limits,
results have also been established for other kinds of countable structures
\cite{masul-countable-chains-2023,barbosa-masul-nenadov-2024,masul-toljic-0}.

A structure $\calS$ is said to have \emph{finite big Ramsey degrees} if
for every finite substructure $\calA$ of $\calS$ there exists a positive integer
$t = t(\calA) \in \NN$ such that the following holds: for every finite Borel coloring
$\gamma : \Emb(\calA, \calS) \to k$ there is an embedding $w : \calS \hookrightarrow \calS$
for which $|\gamma(g \circ \Emb(\calA, \calS))| \le t$.
If no such bound exists for at least one finite substructure of $|calS$,
we say that $\calS$ does not have big Ramsey degrees.

When $\calS$ is a countable relational structure, the Borel $\sigma$-algebra on
$\Emb(\calA, \calS)$ is trivial for every finite substructure $\calA$ of $\calS$,
and consequently every finite coloring $\Emb(\calA, \calS) \to k$ is Borel.
In contrast, when $\calS$ is uncountable, as is the case throughout this paper,
the assumption that the coloring $\gamma$ be Borel becomes essential.

In comparison to the vast body of results concerning big Ramsey degrees of countable relational structures,
our knowledge of the uncountable case remains centered on linear orders and tree-like structures,
largely because the transition to uncountable demands a shift from purely combinatorial techniques
to the more complex set-theoretic toolbox. Ramsey theorems for perfect sets began this line of research,
largely drawing upon Milliken's and Halpern-L\"auchli's theorem.
Blass' proof of Galvin's conjecture on a partition result for perfect subsets of $\RR$ is a
landmark example. The original proof, which was based on the Halpern-L\"auchli theorem, was later simplified
by Todor\v cevi\'c using Milliken spaces (see \cite{todorcevic-RamseySpaces-2010}).

\begin{THM}[Blass \cite{blass-1981}]
  For every perfect subset $P$ of $\RR$ and every finite continuous
  coloring of the set  $[P]^n$ of all $n$-element subsets of $P$,
  there is a perfect set $Q \subseteq P$ such that $[Q]^n$ has at most $(n - 1)!$ colors.
\end{THM}

Other examples of uncountable structures with finite big Ramsey degrees include
the $\kappa$-rationals \cite{DzLM-DLO-2009} and the $\kappa$-Rado graph \cite{DzLM-RADO-2009} for a suitable cardinal $\kappa$,
as well as a large family of universal inverse limit structures \cite{dobrinen-wang-2023}.

In this paper we provide a simple, combinatorial proof of Blass' perfect set theorem,
and derive several consequences regarding a some first-order structures that can
naturally be associated with the Cantor set. While our approach is more accessible, it yields a
weaker result: we establish the existence of big Ramsey degrees for the Cantor set, but do not recover the exact
value of $(n-1)!$ for the degree of an $n$-element subset.

The paper is organized as follows. In Section~\ref{brccs.sec.prelim} we fix the notation and recall
the Dual Ramsey Theorem of Carlson and Simpson, which is out main tool for handling big Ramsey degrees of
the Cantor set and associated structures. In Section~\ref{brccs.sec.CS} we prove a useful corollary of the
Dual Ramsey Theorem and show that the Cantor set with no additional structure has finite big Ramsey degrees.
Finally, in Section~\ref{brccs.sec.CS-FOS} we recall that the Cantor can be endowed with
additional first-order structure in various natural ways, and show that each first-order structure obtained this
way has finite big Ramsey degrees. In particular, we provide a simple proof of Blass' perfect set theorem.
Moreover, we show that the complete Boolean algebra on countably many atoms has finite big Ramsey degrees,
which is in sharp contrast with the recent result \cite{no-fbrd-ctble-atomless-ba} that the countable atomless
Boolean algebra does not have big Ramsey degrees.

\section{Preliminaries}
\label{brccs.sec.prelim}

We adopt the usual set-theoretic notation
$\omega = \{0, 1, 2, 3, \ldots\}$ and $n = \{0, 1, \ldots, n-1\}$, and we let
$\NN = \{1, 2, 3, \ldots \}$. Sometimes it will be more convenient to 
denote the set $B^A$ of all mappings $A \to B$ by $\Map(A, B)$.

Let $(X, \Boxed<)$ and $(Y, \Boxed<)$ be well-ordered sets. A function $f : X \to Y$ is a \emph{rigid surjection}
if $f$ is surjective and $\min f^{-1}(y_1) < \min f^{-1}(y_2)$ for $y_1, y_2 \in Y$
such that $y_1 < y_2$. Let $\RSurj(X, Y)$ denote the set of all rigid surjections $X \to Y$.

Every rigid surjection $f : (X, \Boxed<) \to (Y, \Boxed<)$ uniquely determines an embedding
$f^\partial : (Y, \Boxed<) \hookrightarrow (X, \Boxed<)$ defined by
$f^\partial(y) = \min f^{-1}(y)$, $y \in Y$.
This construction is functorial in the sense that:
\begin{equation}\label{brccs.eq.partial}
  (\id_X)^\partial = \id_X \text{ and } (g \circ f)^\partial = f^\partial \circ g^\partial.
\end{equation}

A \emph{finite coloring} of a set $S$ is a mapping $\gamma : S \to k$ where $k \in \NN$.
A finite coloring $\gamma : S \to k$ of a topological space $S$ is a \emph{Borel coloring}
if $\gamma^{-1}(i)$ is a Borel set, for every $i \in k$.

In the language of rigid surjections, the Infinite Dual Ramsey Theorem of Carlson and Simpson
takes the following form.

\begin{THM}[Dual Ramsey Theorem {\cite[Theorem 1.2]{carlson-simpson-1984}}]
  For every finite linear order $L$ and every finite Borel coloring $\gamma : \RSurj(\omega, L) \to k$
  there is a $w \in \RSurj(\omega, \omega)$ such that $|\gamma(\RSurj(\omega, L) \circ w)| = 1$.
\end{THM}

Let $L$ be a first-order language consisting of functional, relational and constant symbols.
An injective map $\phi : X \to Y$ between first-order structures
$(X, L^X)$ and $(Y, L^Y)$ is a \emph{(first-order) embedding}
if $\phi$ is an isomorphism of $X$ onto the substructure of $Y$ whose underlying set is $\im(\phi) = \{ \phi(x) : x \in X \} \subseteq Y$.
(Note that this implicitly requires that $\im(\phi)$ be closed for fundamental operations $f^Y$ of $Y$ for each functional
symbol $f \in L$, if those exist.)
An injective map $\phi : X \to Y$ between topological spaces $(X, \tau)$ and $(Y, \sigma)$ is a \emph{topological embedding}
if $\phi$ is a homeomorphism of $X$ onto the subspace of $Y$ whose underlying set is $\im(\phi) \subseteq Y$.

A \emph{topological first-order structure} is a topological space $(X, \tau)$ carrying the first order structure
$(X, L^X)$ where the fundamental operations $f^X$ are continuous for each functional symbol $f \in L$,
and relations $R^X$ are closed for each relational symbol $R \in L$.
An injective function $\phi : X \to Y$ between two topological first-order structures is a
\emph{topological first-order embedding} if $\phi$ is both a topological and a first-order embedding of~$X$ into~$Y$.
If $\phi : X \to Y$ is a topological first-order embedding then $\im(\phi) \subseteq Y$ is a
\emph{topological copy of $X$ in $Y$}.

\section{The Cantor set has finite big Ramsey degrees}
\label{brccs.sec.CS}

Using a straightforward corollary of the Dual Ramsey Theorem, in this section we prove that the
Cantor set as a topological space has finite big Ramsey degrees. Our proof is simpler than the proofs of Blass and Todor\v cevi\'c,
but yields a weaker result: we establish the existence of big Ramsey degrees, but do not recover the exact
value of $(n-1)!$ for the degree of an $n$-element subset.

\begin{PROP}\label{brccs.prop.CS}
  For every nonempty finite set $A$ there is a $t = t(|A|) \in \NN$ such that for every finite Borel coloring $\gamma : \Map(\omega, A) \to k$
  there is an $h \in \RSurj(\omega, \omega)$ satisfying $|\gamma(\Map(\omega, A) \circ h)| \le t$.
\end{PROP}
\begin{proof}
  Let $\calT$ denote the set of all nonempty finite words over the alphabet $A$ with non-repeating letters. Put
  $$
    t = |\calT| = \sum_{i=1}^{n} \frac{n!}{(n-i)!}
  $$
  where $n = |A|$. We say that a $\tau \in \calT$ is a \emph{mapping type} of
  $f \in \Map(\omega, A)$, and write $\tp(f) = \tau$, if $f : \omega \to \tau$ is a rigid surjection,
  where we now take $\tau = a_{i_1} a_{i_2} \ldots a_{i_\ell}$ to be a finite
  linear order $a_{i_1} \prec a_{i_2} \prec \ldots \prec a_{i_\ell}$.
  Note that
  $$
    \RSurj(\omega, \tau) = \{f \in \Map(\omega, A) : \tp(f) = \tau\},
  $$
  so $\{\RSurj(\omega, \tau) : \tau \in \calT\}$ is a partition of $\Map(\omega, A)$ into $t$ many blocks.
  It is easy to see that this is a Borel partition of $\Map(\omega, A)$.

  Take any finite Borel coloring $\gamma : \Map(\omega, A) \to k$. We shall now inductively construct a sequence of
  rigid surjections $g_0, g_1, \ldots, g_t \in \RSurj(\omega, \omega)$ and a sequence of
  Borel colorings $\gamma_1$, \ldots, $\gamma_t$ as follows.
  
  To start the induction let $g_0$ be the identity mapping $\omega \to \omega$.
  For the inductive step assume that $g_0$, \ldots, $g_{j-1}$ and $\gamma_1$, \ldots, $\gamma_{j-1}$ have been constructed.
  Define  $\gamma_j : \RSurj(\omega, \tau_j) \to k$ by
  $$
    \gamma_j(f) = \gamma(f \circ g_{j-1} \circ \ldots \circ g_1 \circ g_0).
  $$
  By the Infinite Dual Ramsey Theorem there is a $g_j \in \RSurj(\omega, \omega)$ such that
  $|\gamma_j(\RSurj(\omega, \tau_j) \circ g_j)| = 1$. Therefore,
  $$
    |\gamma(\RSurj(\omega, \tau_j) \circ g_j \circ g_{j-1} \circ \ldots \circ g_1)| = 1.
  $$
  Finally, let $h = g_{t} \circ g_{t-1} \circ \ldots \circ g_1$. Then, having in mind that
  $$
    \RSurj(\omega, L) \circ g \subseteq \RSurj(\omega, L)
  $$
  for every finite linear order $L$ and every $g \in \RSurj(\omega, \omega)$:
  \begin{align*}
    \textstyle|\gamma(\Map(\omega, A) \circ h)| 
      &\textstyle = \sum_{j=1}^t |\gamma(\RSurj(\omega, \tau_j) \circ h)|\\
      &\textstyle = \sum_{j=1}^t |\gamma(\RSurj(\omega, \tau_j) \circ g_t \circ \ldots \circ g_{j+1} \circ g_{j} \circ \ldots \circ g_1)|\\
      &\textstyle \le \sum_{j=1}^t |\gamma(\RSurj(\omega, \tau_j) \circ g_{j} \cdot \ldots \cdot g_1)| = t.
  \end{align*}
  This concludes the proof.
\end{proof}

Let $\calP(A) = 2^A$ denote the \emph{contravariant powerset functor}.
Given a map $f : A \to B$, $\calP(f) : 2^B \to 2^A$ takes $x \in 2^B$
to $x \circ f \in 2^A$. Intuitively, if we take $\calP(A)$ to be the set of all subsets of $A$,
then $\calP(f)$ takes $X \subseteq B$ to $f^{-1}(X) \subseteq A$.

For arbitrary nonempty sets $A$ and $B$, we endow the set $B^A$ with the pointwise-convergence topology,
so that $2^\omega$ becomes the Cantor space.
It is easy to see that $\calP(f) : 2^B \to 2^A$ is continuous for every $\Set$-mapping $f : A \to B$,
and if $f$ is surjective, $\calP(f)$ is a topological embedding.

Since $2^A$ carries the pointwise convergence topology, we endow $(2^A)^B$ with the Tychonoff product topology.
Under this convention, $(2^A)^B$ and $2^{A \times B}$ are homeomorphic. Moreover, the \emph{transpose}:
$$
  \Phi_{A, B} : (2^A)^B \to (2^B)^A \text{ given by } \Phi_{A, B}(f)(a)(b) = f(b)(a)
$$
is a homeomorphism. Note that $\Phi_{A, B}^{-1} = \Phi_{B, A}$.

\begin{LEM}\label{brccs.lem.Phi-0}
  Let $\Phi_{A, B} : (2^A)^B \to (2^B)^A$ be the transpose defined above.
  Take any map $h : A \to A$. Then $\Phi_{A, B}(\calP(h) \circ f) = \Phi_{A, B}(f) \circ h$ for every $f \in (2^A)^B$.
\end{LEM}
\begin{proof}
  Take any $a \in A$ and $b \in B$. Then:
  \begin{multline*}
    \Phi_{A, B}(\calP(h) \circ f)(a)(b) = (\calP(h) \circ f)(b)(a) =\\
	= \calP(h)(f(b))(a) = (f(b) \circ h)(a) = f(b)(h(a)) =\\
	= \Phi_{A, B}(f)(h(a))(b) = (\Phi_{A, B}(f) \circ h)(a)(b).
  \end{multline*}
  This completes the proof of the lemma.
\end{proof}

\begin{THM}\label{brccs.thm.Cantor-0}
  Let $2^\omega$ be the Cantor set. For every nonempty finite set $A$ there is a $t \in \NN$ such that for every finite Borel
  coloring $\gamma : \Map(A, 2^\omega) \to k$ there is an $h \in \RSurj(\omega, \omega)$ such that
  $|\gamma(\calP(h) \circ \Map(A, 2^\omega))| \le t$.
  In particular, there is a Cantor set $\calK \subseteq 2^\omega$ such that $|\gamma(\Map(A, \calK))| \le t$.
\end{THM}
\begin{proof}
  Take any nonempty finite set $A$ and let $t$ be the positive integer provided for $2^A$ by Proposition~\ref{brccs.prop.CS}.
  Take any finite Borel coloring $\gamma : \Map(A, 2^\omega) \to k$, and then define
  $$
    \gamma^* : \Map(\omega, 2^A) \to k \text{\quad by\quad} \gamma^* = \gamma \circ \Phi_{A, \omega},
  $$
  where $\Phi_{-, -}$ is the transpose defined above.
  Since $\gamma^*$ is a Borel coloring, Proposition~\ref{brccs.prop.CS} applies, so there is an $h \in \RSurj(\omega, \omega)$ such that
  $$
    |\gamma^*((2^A)^\omega \circ h)| \le t.
  $$
  Let us show that $|\gamma(\calP(h) \circ \Map(A, 2^\omega))| \le t$. Lemma~\ref{brccs.lem.Phi-0} implies:
  $$
    \Phi_{\omega, A}(\calP(h) \circ \Map(A, 2^\omega)) \subseteq (2^A)^\omega \circ h,
  $$
  whence
  $$
    \gamma^*(\Phi_{\omega, A}(\calP(h) \circ \Map(A, 2^\omega))) \subseteq \gamma^*((2^A)^\omega \circ h).
  $$
  Note that $\gamma^*\circ\Phi_{\omega, A} = \gamma$ because $\Phi_{A, \omega}^{-1} = \Phi_{\omega, A}$. Therefore,
  $$
    |\gamma(\calP(h) \circ \Map(A, 2^\omega))| \le |\gamma^*((2^A)^\omega \circ h)| \le t.
  $$
  Recall that $\calP(h)$ is a topological embedding $2^\omega \hookrightarrow 2^\omega$ because $h$ is surjective. So,
  $|\gamma(\Map(A, \calK))| \le t$ where $\calK = \im(\calP(h)) \subseteq2^\omega$.
\end{proof}

\section{The Cantor set as a topological first-order structure}
\label{brccs.sec.CS-FOS}

Let $\lexle$ denote the \emph{lexicographic} ordering on $2^\omega$. If we identify elements of $2^\omega$ with
subsets of $\omega$ then $\lexle$ can be described as follows:
$A \lexle B$ iff $A \subseteq B$, or $\min(B \setminus A) < \min(A \setminus B)$ in case $A$ and $B$ are incomparable.

Let $\Emb(n, 2^\omega)$ denote the set of all the embeddings of the finite linear order $n = \{0, 1, \ldots, n-1\}$
into $(2^\omega, \Boxed{\lexle})$. Note that $\Emb(n, 2^\omega)$ is just a convenient representation of
the set $[2^\omega]^n$ of all $n$-element subsets of $2^\omega$.

\begin{THM}\label{brccs.thm.Cantor-n}
  Let $2^\omega$ be the Cantor set ordered linearly by~$\lexle$.
  For every $n \in \NN$ there is a $t \in \NN$ such that for every finite Borel coloring $\gamma : \Emb(n, 2^\omega) \to k$
  there is a topological first-order embedding $g : (2^\omega, \Boxed\lexle) \hookrightarrow (2^\omega, \Boxed\lexle)$ such that
  $|\gamma(g \circ \Emb(n, 2^\omega))| \le t$.
  In particular, there is a Cantor set $\calK \subseteq 2^\omega$ such that $|\gamma(\Emb(n, \calK))| \le t$.
\end{THM}
\begin{proof}
  Take any $n \in \NN$ and let $t = t(n)$ be the positive integer provided by Theorem~\ref{brccs.thm.Cantor-0}.
  Take any finite Borel coloring $\gamma : \Emb(n, 2^\omega) \to k$, and then define
  $\gamma' : \Map(n, 2^\omega) \to k$ by
  \begin{align*}
    \gamma'(a) &= \gamma(a), \text{ for } a \in \Emb(n, 2^\omega)\\
	\gamma'(x) &= 0, \text{ otherwise}.
  \end{align*}
  This is a Borel coloring because $\Emb(n, 2^\omega)$ is a Borel subset of $\Map(n, 2^\omega)$.
  Theorem~\ref{brccs.thm.Cantor-0} then tells us that there is an $h \in \RSurj(\omega, \omega)$ such that
  $$
    |\gamma'(\calP(h) \circ \Map(n, 2^\omega))| \le t.
  $$

  Let $g = \calP(h)$. We have seen that $g$ is a topological embedding. Let us show that this is also
  a relational embedding $(2^\omega, \Boxed\lexle) \hookrightarrow (2^\omega, \Boxed\lexle)$.
  For this part of the proof it will be more convenient to take $\calP(\omega)$ to be the set of all the subsets of $\omega$.
  Then $g(A) = h^{-1}(A)$ for $A \subseteq \omega$.

  \medskip

  Claim 1. $\min g(A) = \min h^{-1}(\min A)$ for every $\0 \ne A \in \calP(\omega)$.

  Proof. Let $s = \min A$. Then $\min h^{-1}(s) < \min h^{-1}(a)$ for all $a \in A \setminus \{s\}$ because $h$ is a rigid surjection.
  Therefore, $\min g(A) = \min\big(h^{-1}(s) \cup \bigcup_{a \in A \setminus \{s\}} h^{-1}(a)\big) = \min h^{-1}(s)$. This proves Claim~1.

  \medskip

  Claim 2. $A \lexle B$ iff $g(A) \lexle g(B)$ for all $A, B \in \calP(\omega)$.

  Proof. Note that $A = \0$ iff $g(A) = h^{-1}(A) = \0$, so the claim holds by the definition of $\lexle$ if one of $A$, $B$ is empty. For the rest of the proof
  assume that $A, B \ne \0$. If $A \subseteq B$ then $h^{-1}(A) \subseteq h^{-1}(B)$ trivially. Assume, therefore, that $A$ and $B$ are incomparable.
  Then $A \lexle B$ means that $\min(B \setminus A) < \min(A \setminus B)$. Claim~1 and the fact that $h$ is a rigid surjection now yield:
  $
    \min g(B \setminus A) = \min h^{-1}(\min (B \setminus A)) < \min h^{-1}(\min (A \setminus B)) = \min g(A \setminus B) 
  $.
  This proves Claim 2.
  
  \medskip
  
  Since $g$ is a topological first-order embedding $(2^\omega, \Boxed\lexle) \hookrightarrow (2^\omega, \Boxed\lexle)$, it follows that
  $g \circ \Emb(n, 2^\omega) \subseteq \Emb(n, 2^\omega)$. Therefore,
  $$
    \gamma(g \circ \Emb(n, 2^\omega)) = \gamma'(g \circ \Emb(n, 2^\omega)) \subseteq \gamma'(g \circ \Map(n, 2^\omega))
  $$
  showing that $|\gamma(g \circ \Emb(n, 2^\omega))| \le t$.
  In other words, $|\gamma(\Emb(n, \calK))| \le t$ where $\calK = \im(g) \subseteq 2^\omega$.
\end{proof}

\begin{COR}[cf.\ Blass \cite{blass-1981}]
  For every $n \in \NN$ there is a $t \in \NN$ such that for every perfect set $P \subseteq \RR$ and every
  finite Borel coloring $\gamma : [P]^n \to k$ of strictly increasing $n$-tuples of elements of $P$
  there is a perfect set $Q \subseteq P$ such that $|\gamma([Q]^n)| \le t$.
\end{COR}
\begin{proof}
  Take any $n \in \NN$ and let $t$ be the positive integer provided for $n$ by Theorem~\ref{brccs.thm.Cantor-n}.
  Take any perfect $P \subseteq \RR$ and a finite Borel coloring $\gamma : [P]^n \to k$.
  Let $\calC \subseteq P$ be a homeomorph of the Cantor set contained in $P$ and let
  $
    \gamma' : \Emb(n, \calC) \to k
  $
  be the restriction of $\gamma$ to $[\calC]^n = \Emb(n, \calC)$. By Theorem~\ref{brccs.thm.Cantor-n}
  there is a Cantor set $\calK \subseteq \calC$ such that $|\gamma'(\Emb(n, \calK))| \le t$. So we can take $Q = \calK$.
\end{proof}

Besides the liner order $\lexle$, the Cantor set $2^\omega$
carries the obvious partial order $\le$ defined coordinatewise ($f \sqsubseteq g$ iff $f(n) \le g(n)$ for all $n \in \omega$)
turning $2^\omega$ into a topological poset $\calQ = (2^\omega, \Boxed\sqsubseteq)$.
Note that $\lexle$ is a linear order of $2^\omega$ that extends $\sqsubseteq$, so $\calE\calQ = (2^\omega, \Boxed\sqsubseteq, \Boxed\lexle)$
is a \emph{linearly ordered poset} (that is, a poset together with a linear order which extends the partial order).
The Cantor set also carries a natural graph structure:
for $f, g \in 2^\omega$ we let $f \sim g$ if $f \ne g$ and there is an $n \in \omega$ such that $f(n) = g(n) = 1$.
Let $\calG = (2^\omega, \Boxed\sim)$ be the corresponding topological graph, and let
$\calO\calG = (2^\omega, \Boxed\sim, \Boxed\lexle)$ denote its ordered version.

On top of that, it is easy to check that the join $\lor$, meet $\land$ and
the operation $(-)^c$ of taking the complement (defined coordinatewise by $(f \lor g)(n) = \max\{f(n), g(n)\}$, $f, g \in 2^\omega$,
and analogously for $\land$ and $(-)^c$) are continuous with respect to the topology of the Cantor set.
Therefore, $\calB = (2^\omega, \Boxed\land, \Boxed\lor, \Boxed{(-)^c}, 0, 1)$ is a topological Boolean algebra.
Again, $\lexle$ extends the partial order inherent to the Boolean structure, so
$\calE\calB = (2^\omega, \Boxed\land, \Boxed\lor, \Boxed{(-)^c}, 0, 1, \Boxed\lexle)$ is a \emph{linearly ordered topological Boolean algebra}.
Note that $\calL = (2^\omega, \Boxed\land, \Boxed\lor)$ is a topological Boolean lattice and that
$\calE\calL = (2^\omega, \Boxed\land, \Boxed\lor, \Boxed\lexle)$ is a \emph{linearly ordered topological Boolean lattice}.

The proof of Theorem~\ref{brccs.thm.Cantor-n} gives us a blueprint for the proof of the following statement.

\begin{THM}
  Let $\calC$ be any of the eight structures based on the Cantor set introduced above. For every finite substructure $\calA$ of $\calC$
  there is a $t \in \NN$ such that for every finite Borel coloring $\gamma : \Emb(\calA, \calC) \to k$
  there is a topological first-order embedding $g : \calC \hookrightarrow \calC$ such that
  $|\gamma(g \circ \Emb(\calA, \calC))| \le t$.
  In particular, there is a topological copy $\calK \subseteq \calC$ of $\calC$ such that $|\gamma(\Emb(\calA, \calK))| \le t$.
\end{THM}
\begin{proof}
  Let $L$ be the first-order language of $\calC$. Take a finite substructure $\calA = (A, L^A)$ of $\calC$,
  and let $t$ be the positive integer provided by Theorem~\ref{brccs.thm.Cantor-0} for~$A$.
  Take any finite Borel coloring $\gamma : \Emb(\calA, \calC) \to k$, and then define
  $\gamma' : \Map(A, 2^\omega) \to k$ by:
  \begin{align*}
    \gamma'(a) &= \gamma(a), \text{ for } a \in \Emb(\calA, \calC)\\
	\gamma'(x) &= 0, \text{ otherwise}.
  \end{align*}
  This is a Borel coloring because $\Emb(\calA, \calC)$ is a Borel subset of $\Map(A, 2^\omega)$.
  Theorem~\ref{brccs.thm.Cantor-0} then tells us that there is an $h \in \RSurj(\omega, \omega)$ such that
  $$
    |\gamma'(\calP(h) \circ \Map(A, 2^\omega))| \le t.
  $$

  Let $g = \calP(h)$. It is easy to show that $g$ is a topological first-order embedding $\calC \hookrightarrow \calC$, so
  $g \circ \Emb(\calA, \calC) \subseteq \Emb(\calA, \calC)$. Therefore,
  $$
    \gamma(g \circ \Emb(\calA, \calC)) = \gamma'(g \circ \Emb(\calA, \calC)) \subseteq \gamma'(g \circ \Map(A, 2^\omega)),
  $$
  showing that $|\gamma(g \circ \Emb(\calA, \calC))| \le t$.
  In other words, $|\gamma(\Emb(\calA, \calK))| \le t$ where $\calK = \im(g) \subseteq \calC$.
\end{proof}

In particular, this result tells us that the topological Boolean algebra $\calB = (2^\omega, \Boxed\land, \Boxed\lor, \Boxed{(-)^c}, 0, 1)$
has finite big Ramsey degrees.
By contrast, it has recently been shown \cite{no-fbrd-ctble-atomless-ba} that the countable atomless Boolean algebra does not have big Ramsey degrees.
This sharp distinction suggests the following:

\paragraph{Open problem.} Classify Boolean algebras with finite big Ramsey degrees.

\end{document}